\documentclass[11pt]{amsart}

\usepackage{amssymb,amsthm,amsmath}
\usepackage[numbers,sort&compress]{natbib}
\usepackage[paper=a4paper, left=2.35cm, right=2.35cm, top=2.5cm, bottom=1.8cm, headheight=5.5pt, footskip=0.8cm, footnotesep=0.8cm, centering, includefoot]{geometry}

\title[Nonlocal Kirchhoff equations with lack of compactness]{Nonlocal Kirchhoff superlinear equations with  indefinite nonlinearity and lack of compactness}
\author[L. Li, V. R\u{a}dulescu, and D. Repov\v{s}]{Lin Li, Vicen\c{t}iu R\u{a}dulescu, and Du\v{s}an Repov\v{s}}
\address[L. Li]{School of Mathematics and Statistics, Chongqing Technology and Business University,  400067 Chongqing, P.R. China}
\email{lilin420@gmail.com}
\address[V. R\u{a}dulescu]{Institute of Mathematics, Physics and Mechanics, 1000 Ljubljana, Slovenia
\&  Department of Mathematics, University of Craiova,  200585 Craiova,
Romania. } \email{vicentiu.radulescu@imfm.si}
\address[D. Repov\v{s}]{Faculty of Education, and Faculty of Mathematics and Physics,
University of Ljublja\-na
\&
Institute of Mathematics, Physics and Mechanics, 1000 Ljubljana, Slovenia}\email{dusan.repovs@guest.arnes.si}


\newcommand{\RR}{\mathbb{R}}

\theoremstyle{plain}
\newtheorem{theorem}{Theorem}[section]

\newtheorem{lemma}[theorem]{Lemma}

\theoremstyle{definition}

\newtheorem{remark}[theorem]{Remark}

\begin{document}


\begin{abstract}
We study the following Kirchhoff equation
\begin{equation*}
- \left( 1 + b \int_{\RR^3} |\nabla u|^2 dx \right) \Delta u + V(x) u = f(x,u), \quad x \in \RR^3.
\end{equation*}
A feature of this paper is that the nonlinearity $f$ and the potential $V$ are indefinite, hence sign-changing. Under some appropriate assumptions on $V$ and $f$, we prove the existence of two different solutions of the equation via the Ekeland variational principle and the mountain pass theorem.
\end{abstract}
\maketitle

\section{Introduction}
In this paper we consider the following Kirchhoff equation
\begin{equation}\label{prob}\tag{$\mathcal{K}$}
- \left( 1 + b \int_{\RR^3} |\nabla u|^2 dx \right) \Delta u + V(x) u = f(x,u), \quad x \in \RR^3,
\end{equation}
where $b $ is a positive constant, the potential $V$ and the nonlinearity $f$ are allowed to be sign-changing.

Equation \eqref{prob} is a modified version of the classical Kirchhoff equation, which has a strong physical meaning. Problem \eqref{prob} is related to the stationary analogue of the Kirchhoff equation
\begin{equation}\label{eq2}
u_{tt}-\left(1 + b \int_{\Omega}|\nabla_x u|^2dx \right) \Delta_x u=g(x,u)
\end{equation}
which was proposed by Kirchhoff \cite{K} in 1883.
The early classical studies of the Kirchhoff equation were made by Bernstein \cite{B} and Poho\~{z}aev \cite{P}. However, Equation \eqref{eq2} received great attention only after  Lions \cite{L} proposed an abstract framework for the problem.

 The Kirchhoff equation is a generalization of the d'Alembert wave equation
\begin{equation*}
    \rho \frac{\partial ^2u}{\partial t^2}-\left(\frac{P_0}{h} + \frac{E}{2L}\int_0^L\left|\frac{\partial u}{\partial x}\right|^2dx\right)\frac{\partial ^2u}{\partial x^2}=g(x,u)
\end{equation*}
for free vibrations of the elastic string. Kirchhoff's model takes into account the changes in the length of the string produced by transverse vibrations. Here, $L$ is the length of the string, $h$ is the area of its cross section, $E$ is the Young modulus of the material, $\rho$ is the mass density and $P_0$ is the initial tension.  It was pointed out in \cite{alves} that Eq. \eqref{eq2} models various physical phenomena, where $u$ describes a process that depends on the average of itself. Nonlocal effects also arise in the description of biological systems. A parabolic version of problem \eqref{eq2} can be used to describe the growth and movement of some species. In this case, the integral term models the movement, which is assumed to be dependent on the energy of the
entire system with the unknown $u$
being its population density.

 We focus on the Euclidean space 3-space with lack of compactness, since the Sobolev embedding is not compact for the whole space. A natural idea is study this equation on the radial space. Interested reader can consult the references \cite{MR2651702,MR2825259,MR2951740,MR2946973,MR3359229,MR3128653}. Recently, Wu \cite{MR2736309} studied this type of equations with positive coercive potential $V$. Four new existence results for nontrivial solutions and a sequence of high
energy solutions for problem \eqref{prob} were obtained by using a symmetric mountain pass theorem. Actually, coercive potential $V$ was introduced by Rabinowitz \cite{MR1162728} (see also \cite{gazrad}) to overcome the lack of compact Sobolev embedding. Later, many authors \cite{MR2914487,MR3053562,MR3004514,MR3341384,MR3091217,MR3135452,MR3169022,MR3162387,MR3239677,MR3372045,MR3316612,MR3365848} used this type of potential.
Very recently, the case when the potential $V$ vanishes at some points has also been considered \cite{MR3158465,MR3191413,MR3334964,MR3347399,MR3365866}. We also refer to the related papers \cite{ore1,molrep,ore2,pucci1,pucci2} and the monograph \cite{molrads}, which deals with variational methods for nonlocal fractional equations.

In some of the aforementioned references, the potential $V$ is always assumed to be positive or vanishing at infinity. The following technical Ambrosetti--Rabinowitz  condition ((AR) for short) is usually required.
\begin{itemize}
\item[(AR)] There exists $\mu > 4$ such that $$ 0 < \mu F(x,u) \leq u f(x,u), \quad u \not=0. $$
\end{itemize}

 The role of (AR) is to ensure the boundedness of the Palais--Smale (PS) sequences of the energy
functional, which is  crucial in applying the critical point theory.

Motivated by the works \cite{MR3011282,MR3283380},  we consider in this paper another case, namely that of $f$ being superlinear, that is, $f ( x , u )/ u \to +\infty $ as $u \to \infty $. Furthermore, the potential $V$ and the primitive of $f$ are also allowed to be sign-changing, which is quite different
from the previous results. Before stating our main results, we list the following assumption on $V ( x ) $.
\begin{itemize}
\item[(V1)] $V \in C(\RR^3 , \RR)$ and $\inf_{x \in \RR^3} V(x) > -\infty$. Moreover, there exists a constant $d_0 > 0$ such that for any $M > 0$, $$ \lim_{|y|\to\infty} \text{meas} \left( x \in \RR^3 : |x-y|\leq d_0, V(x) \leq M \right) = 0, $$ where meas $(\cdot)$ denotes the Lebesgue measure in $\RR^3 $.
\end{itemize}

Inspired by Zhang and Xu \cite{MR2769179}, we can find a constant $V_0 > 0$ such that $\tilde{V}(x) := V ( x ) + V_0 \geq 1$ for all $x \in \RR^3$ and let $\tilde{f} ( x , u ) := f ( x , u ) + V_0 u $, $\forall( x , u ) \in \RR^3 \times \RR $. Then it is easy to verify the following lemma.
\begin{lemma}\label{lemma1.1}
Equation \eqref{prob} is equivalent to the following problem
\begin{equation}\label{kgm}\tag{$\mathcal{K'}$}
- \left( 1 + b \int_{\RR^3} |\nabla u|^2 dx \right) \Delta u + \tilde{V}(x) u = \tilde{f}(x,u), \quad x \in \RR^3.
\end{equation}
\end{lemma}

In what follows, we let $\mu > 4$ and impose some assumptions on $\tilde{f}$ and its primitive $\tilde{F}$ as follows:
\begin{itemize}
\item[(S1)] $\tilde{f} \in C ( \RR^3 \times \RR , \RR ) $, and there exist constants $c_1 $, $c_2 > 0$ and $q \in ( 4 , 6 )$ such that $$ |\tilde{f}(x,u)| \leq c_1|u|^3 + c_2|u|^{q-1}. $$
\item[(S2)] $\lim_{|u|\to\infty} \frac{|\tilde{F}(x,u)|}{|u|^4} = \infty$ a.e. $x \in \RR^3$ and there exist constants $c_3\geq 0$, $r_0\geq 0$ and $\tau \in (0,2)$ such that $$ \inf_{x\in\RR^3} \tilde{F}(x,u) \geq c_3 |u|^\tau \geq 0, \quad \forall (x,u)\in \RR^3\times\RR,\, |u|\geq r_0, $$ where (and in the sequel) $\tilde{F}(x,u) = \int_0^u \tilde{f} (x,s)ds$.
\item[(S3)] $\tilde{\mathcal{F}}(x,u):=\frac{1}{4}u\tilde{f}(x,u)-\tilde{F}(x,u)\geq 0$, and there exist $c_4 > 0$ and $\kappa > 1$ such that $$ |\tilde{F}(x,u)|^\kappa \leq c_4 |u|^{2\kappa} \tilde{\mathcal{F}}(x,u), \quad \forall (x,u)\in \RR^3\times\RR,\, |u|\geq r_0. $$
\end{itemize}

Now we state our main result as follows.
\begin{theorem}\label{thm1.1}
Suppose that conditions (V1), (S1), (S2) and (S3) are satisfied. Then problem \eqref{prob} has at least two different solutions.
\end{theorem}

\begin{remark}
There are some functions not satisfying the condition (AR). For example, the superlinear function
$f ( x , u ) = \sin x \ln ( 1 + | u |) u^2$ does not satisfy condition (AR). In our theorems, $\tilde{F} ( x , u )$ is allowed to be sign-changing. Even if $\tilde{F} ( x , u ) \geq  0$, the assumptions (S2) and (S3) seem to be weaker than the superlinear conditions obtained in the aforementioned references. By straightforward computation we check that the following nonlinearity $\tilde{f}$ satisfies (S2) and (S3):
\[ \tilde{f}(x,u) = a(x) (4u^4 + 2u^2 \sin u - 4u\cos u) \] where $a \in (\RR^3,\mathbb{R})$ and $ 0 < \inf_{\RR^3} a(x) \leq \sup_{\RR^3} a(x) < \infty$.
\end{remark}

\begin{remark}
To the best of our knowledge, the condition (V1) was first stated in \cite{MR2181481}, but $\inf_{x\in \RR^3} V ( x ) > 0$ was required. From (V1), one can see that the potential $V ( x )$ is allowed to be sign-changing. Therefore, the condition (V1) is weaker than those in \cite{MR2914487,MR3053562,MR3004514,MR3341384,MR3091217,MR3135452,MR3169022,MR3162387,MR3239677,MR3372045,molradccm,MR3316612,MR3365848,MR3158465,MR3191413,MR3334964,MR3347399,MR3365866}.
\end{remark}
\begin{remark}
It is not difficult to find the functions $V$ satisfying the above conditions. For example, let $V ( x )$ be a zig–zag function with respect to $| x |$ defined by
\begin{equation*}
V(x)=
\begin{cases}
2n|x|-2n(n-1)+a_0, & n-1 \leq |x| \leq (2n-1)/2,\\
-2n|x|+2n^2+a_0, & (2n-1)/2 \leq |x| \leq n,
\end{cases}
\end{equation*}
where $n \in \mathbb{N}$ and $a_0 \in \RR $.
\end{remark}
\begin{remark}
Zhang {\it et al.} \cite{MR3239677} studied \eqref{prob} with sign-changing potential $V$. They obtained multiple solutions in the case of odd nonlinearity. Here we do not need that the nonlinearity is odd and we also get two solutions for problem \eqref{prob}. Bahrouni \cite{MR3065051} obtained infinitely many solutions for \eqref{prob} with the potential and nonlinearity  both sign-changing. However, he studied the sublinear case and with odd nonlinearity. Here our results can be regarded as an extension of the results of \cite{MR3239677,MR3065051}.
\end{remark}

\section{Preliminaries and variational setting}
Hereafter, we use the following notation:
\begin{itemize}
\item $H^1 ( \RR^3 )$ denotes the usual Sobolev spaces endowed with the standard scalar product and norm $$ (u,v) = \int_{\RR^3} (\nabla u \cdot \nabla v + uv)dx, \quad \|u\| = (u,u)^{1/2}. $$
\item $D^{1,2} ( \RR^3 )$ denotes the completion of $C_0^\infty ( \RR^3 )$ with respect to the norm  $$ \|u\|^2_{D^{1,2}(\RR^3)} = \int_{\RR^3} |\nabla u|^2 dx. $$
\item  $H = \left\{ u \in H^1 ( \RR^3 ) : \int_{\RR^3} (|\nabla u |^2 + \tilde{V} ( x )| u |^ 2 ) dx < \infty \right\}$ is the Sobolev space endowed with the norm  $$ \|u\|^2_H ( \RR^3 )  = \int_{\RR^3} (|\nabla u |^2 + \tilde{V} ( x )| u |^ 2 ) dx.$$
\item $H^*$ denotes the dual space of $H$.
\item $L^s ( \RR^3 )$, $1 \leq s < +\infty$ denotes a Lebesgue space with the usual norm $\| u \|_s = \left( \int_{\RR^3} | u |^s dx \right)^{1/s}$.
\item For any $\rho > 0$ and for any $z \in \RR^3 $, $B_\rho ( z )$ denotes the ball of radius $\rho$ centered at $z$.
\item $C$ and $C_i$ denote various positive constants, which may vary from line to line.
\item $S_i$ denote the Sobolev constant for the embedding.
\item $\to$ denotes the strong convergence and $\rightharpoonup$ denotes the weak convergence.
\end{itemize}
Throughout this section, we make the following assumption instead of (V1):
\begin{itemize}
\item[(V2)] $\tilde{V} \in C(\RR^3 , \RR)$ and $\inf_{x \in \RR^3} \tilde{V}(x) > 0$. Moreover, there exists a constant $d_0 > 0$ such that for any $M > 0$, $$ \lim_{|y|\to\infty} \text{meas} \left\{ x \in \RR^3 : |x-y|\leq d_0, V(x) \leq M \right\} = 0. $$
\end{itemize}
\begin{remark}\label{remark2.1}
Under assumptions (V2), we know by Lemma 3.1 in \cite{MR2181481} that the embedding $H \hookrightarrow L^s ( \RR^3 )$ is compact for $s \in [ 2 , 6 ) $.
\end{remark}

Let $I:H\rightarrow \RR$ denote the energy functional defined by
\begin{equation}\label{eq:2.5}
    I (u) = \frac{1}{2} \int_{\RR^3}(|\nabla u|^2 + \tilde{V}(x)u^2 )  dx + \frac{b}{4} \left( \int_{\RR^3} |\nabla u|^2 dx \right)^2 - \int_{\RR^3} \tilde{F}(x,u) dx,
\end{equation}
for all $u \in H$. By condition (S1), we have
\begin{equation}\label{eq:2.6}
|\tilde{F}(x,u)| \leq \frac{c_1}{4}|u|^4+\frac{c_2}{q}|u|^q,\quad \forall (x,u)\in \RR^3\times\RR.
\end{equation}
Consequently, similar to the discussion in \cite{MR2736309}, under assumptions (V2) and \eqref{eq:2.6},  the
functional $I$ is of class $C^1 ( H , \RR ) $. Moreover,
\begin{equation}\label{eq:2.7}
\langle I'(u), v \rangle = \left( 1 + b \int_{\RR^3} |\nabla u|^2 dx \right) \int_{\RR^3} \nabla u \cdot \nabla v dx + \int_{\RR^3} \tilde{V}(x) uv dx  - \int_{\RR^3} \tilde{f}(x,u)v dx.
\end{equation}
Hence, if $u \in H$ is a critical point of $I$, then $u$ is a solution of equation \eqref{kgm}.

We now recall the mountain pass theorem of Ambrosetti and Rabinowitz \cite{ambrab} without the Palais-Smale condition (see also \cite{MR1051888}).
We also refer to Brezis and Nirenberg \cite{brenir} for a simple proof of this result which uses the Ekeland variational principle in combination with a pseudo-gradient argument.

\begin{lemma}\label{lemma2.1}
Let $E$ be a real Banach space with its dual space $E^*$, and suppose that $I \in C^1 ( E , \RR )$ satisfies $$\max\{I(0),I(e)\}\leq \mu < \eta \leq \inf_{\|u\|=\rho}I(u),$$ for some $\mu$, $\eta$, $\rho > 0$ and $e \in E$ with $\| e \| > \rho $. Let $c \geq \eta$ be characterized by  $$ c = \inf_{\gamma \in \Gamma}\max_{0\leq\tau\leq1} I(\gamma(\tau)), $$ where $\Gamma = \{ \gamma \in C ([ 0 , 1 ], E ) : \gamma ( 0 ) = 0 , \gamma ( 1 ) = e \}$ is the set of all continuous paths joining 0 and $e$. Then there exists a sequence $\{ u_n \} \subset E$ such that
$$ I(u_n) \to c \geq \eta \text{ and }(1+\|u_n\|)\|I'(u_n)\|_{E^*} \to 0, \text{ as }n \to \infty.$$
\end{lemma}

This kind of sequence is usually called a Cerami sequence. Recall that a $C^1$ functional $I$ satisfies the Cerami compactness condition at level $c$ ($(C)_c$ condition for short) if any sequence $\{ u_n \} \subset H$ such that $I ( u_n ) \to c$ and $( 1 + \| u_n \|)\| I'  ( u_n )\|_{E^*}  \to 0$ has a convergent subsequence.

Here, we give the sketch of how to look for two distinct critical points of the functional $I$. First, we consider a minimization
of $I$ constrained to a neighborhood of zero via the Ekeland variational principle (see \cite{MR0346619,MR1400007}) and we can find a critical point
of $I$ which achieves the local minimum of $I$ and the level of this local minimum is negative (see  Step 1 of the proof of
Theorem \ref{thm1.1}). Next, around the ``zero'' point, by using mountain pass theorem (see \cite{ambrab}), we obtain a second critical
point of $I$ with its positive level (see Step 2 of the proof of Theorem \ref{thm1.1}). Obviously, these two critical points do not coincide
since they have different energy levels.

To prove Theorem \ref{thm1.1}, we cite the following auxiliary result, see \cite{MR3011282}.
\begin{lemma}\label{tang}
  Assume that $p_1$, $p_2 > 1$, $r$, $q \geq 1$ and $\Omega \subseteq \RR^3$. Let $g(x,t)$ be a Carath\'{e}odory function on $\Omega \times \RR$  satisfying
  \begin{equation*}
    |g(x,t)| \leq a_1 |t|^{(p_1 - 1)/r} + a_2 |t|^{(p_2 -1)/r}, \quad \forall (x,t) \in \Omega\times\RR,
  \end{equation*}
  where $a_1$, $a_2 \geq 0$. If $u_n \to u$ in $L^{p_1} (\Omega) \cap L^{p_2}(\Omega)$, and $u_n \to u$ a.e. $x \in \Omega$, then for any $v \in L^{p_1q}(\Omega) \cap L^{p_2q}(\Omega)$,
  \begin{equation}\label{eq:22.14}
    \lim_{n \to \infty} \int_\Omega |g(x,u_n) - g(x,u)|^{r}|v|^{q} dx \to 0.
  \end{equation}
\end{lemma}
\section{Proof of the main result}

\begin{lemma}\label{lemma2.2}
Assume that the conditions (V2) and (S1) hold. Then there exist $\rho$, $\eta > 0$ such that $\inf \{ I ( u ) : u \in H \text{ with } \| u \|_H = \rho\} > \eta $.
\end{lemma}
\begin{proof}
By \eqref{eq:2.6} and the Sobolev inequality, we have
\begin{align}\label{eq:2.9}
\left| \int_{\RR^3} \tilde{F}(x,u) dx \right| & \leq \int_{\RR^3} \left| \frac{c_1}{4}|u|^4 + \frac{c_2}{q}|u|^q \right| dx\\\nonumber
& = \frac{c_1}{4}\|u\|_4^4 + \frac{c_2}{q}\|u\|_q^q\\\nonumber
& \leq S_4 \frac{c_1}{4}\|u\|_H^4 + S_q \frac{c_2}{q}\|u\|_H^q,
\end{align}
for any $u \in H$. Combining \eqref{eq:2.5} with \eqref{eq:2.9}, we obtain
\begin{align}\label{eq:2.10}
I(u) & = \frac{1}{2} \int_{\RR^3}(|\nabla u|^2 + \tilde{V}(x)u^2 )  dx + \frac{b}{4} \left( \int_{\RR^3} |\nabla u|^2 dx \right)^2 - \int_{\RR^3} \tilde{F}(x,u) dx\\\nonumber
& \geq \frac{1}{2} \|u\|_H^2 - \int_{\RR^3} | \tilde{F}(x,u) | dx\\\nonumber
& \geq \frac{1}{2} \|u\|_H^2 -  S_4 \frac{c_1}{4}\|u\|_H^4 - S_q \frac{c_2}{q}\|u\|_H^q\\\nonumber
& = \frac{1}{2} \|u\|_H^2 -  C_1 {4}\|u\|_H^4 - C_2 \|u\|_H^q.
\end{align}
Since $q \in ( 4 , 6 ) $, we deduce that there exists $\eta > 0$ such that this lemma holds if we let $\|u\|_H = \rho > 0$ be small enough.
\end{proof}

\begin{lemma}\label{lemma2.3}
Assume that the conditions (V2) and (S2) hold. Then there exists $v \in H$ with $\|v\|_H = \rho$ such that $I (v) < 0$, where $\rho$ is given in Lemma \ref{lemma2.2}.
\end{lemma}
\begin{proof}
By \eqref{eq:2.5}, we have
$$ \frac{I(tu)}{t^4} = \frac{1}{2t^2} \|u\|^2_H + \frac{b}{4} \left( \int_{\RR^3} |\nabla u|^2 dx \right)^2 - \frac{1}{t^4} \int_{\RR^3} \tilde{F}(x,tu) dx. $$
Then, by (S2) and Fatou's lemma we can deduce that
\begin{align*}
\lim_{t \to \infty}\frac{I(tu)}{t^4} & = \lim_{t \to \infty} \left[ \frac{1}{2t^2} \|u\|^2_H + \frac{b}{4} \left( \int_{\RR^3} |\nabla u|^2 dx \right)^2 - \frac{1}{t^4} \int_{\RR^3} \tilde{F}(x,tu) dx \right]\\
& = \limsup_{t \to \infty} \left[ \frac{b}{4} \left( \int_{\RR^3} |\nabla u|^2 dx \right)^2 - \frac{1}{t^4} \int_{\RR^3} \tilde{F}(x,tu) dx \right] \\
& \leq \frac{b}{4} \left( \int_{\RR^3} |\nabla u|^2 dx \right)^2 -   \liminf_{t \to \infty} \int_{\RR^3} \frac{\tilde{F}(x,tu)}{t^4u^4} u^4 dx\\
& \leq C_3 \|u\|_H^4 - \int_{\RR^3} \liminf_{t \to \infty} \frac{\tilde{F}(x,tu)}{t^4u^4} u^4 dx\\
& = - \infty \text{ as }t \to \infty.
\end{align*}
Thus the lemma is proved by taking $v = t_0 u$ with  large enough $t_0$.
\end{proof}
Based on Lemmata \ref{lemma2.2} and \ref{lemma2.3}, Lemma \ref{lemma2.1} implies that there is a sequence $\{ u_n \} \subset H$ such that
\begin{equation}\label{eq:2.12}
I(u_n) \to c >0 \text{ and }(1+\|u_n\|_H)\|I'(u_n)\|_{H^*} \to 0, \text{ as }n \to \infty.
\end{equation}
\begin{lemma}\label{lemma2.4}
Assume that the conditions (V2), (S1), (S2) and (S3) hold. Then the sequence $\{ u_n \}$ defined in \eqref{eq:2.12} is bounded in $H$.
\end{lemma}
\begin{proof}
Arguing by contradiction, we can assume $\| u_n \|_H \to \infty $. Define $v_n := \frac{u_n}{\|u_n\|_H} $. Clearly, $\| v_n \|_H = 1$ and $\| v_n \|_s \leq S_s \|v_n\|_H = S_s$, for $2 \leq s < 6$. Observe that for large enough $n$, we can get from \eqref{eq:2.12} and (S3) that
\begin{align}\label{eq:2.13}
c+1  & \geq I(u_n) - \frac{1}{4}\langle I'(u_n), u_n \rangle\\\nonumber
& = \frac{1}{4} \|u_n\|_H^2 + \int_{\RR^3} \left( \frac{1}{4} \tilde{f}(x,u_n)u_n - \tilde{F}(x,u_n) \right) dx\\\nonumber
& \geq \int_{\RR^3} \tilde{\mathcal{F}}(x,u_n)dx.
\end{align}
In view of \eqref{eq:2.5} and \eqref{eq:2.12}, we have
\begin{align}\label{eq:2.14}
\frac{1}{2} & = \frac{I(u_n)}{\|u_n\|_H^2}+\frac{1}{\|u_n\|_H^2}\int_{\RR^3}\tilde{F}(x,u_n)dx-\frac{b}{4\|u_n\|_H^2} \left( \int_{\RR^3} |\nabla u|^2 dx \right)^2\\\nonumber
& \leq \frac{I(u_n)}{\|u_n\|_H^2} + \frac{1}{\|u_n\|_H^2}\int_{\RR^3}|\tilde{F}(x,u_n)|dx\\\nonumber
& \leq \limsup_{n\to\infty}\left[ \frac{I(u_n)}{\|u_n\|_H^2} + \frac{1}{\|u_n\|_H^2}\int_{\RR^3}|\tilde{F}(x,u_n)|dx \right]\\\nonumber
& \leq \limsup_{n\to\infty}\int_{\RR^3}\frac{|\tilde{F}(x,u_n)|}{\|u_n\|_H^2}dx.
\end{align}
For $0 \leq a < b$, let $\Omega_n ( a , b ) := \{ x \in \RR^3 : a \leq | u_n ( x )| < b \} $. Going if necessary to a subsequence, we may assume that $v_n \rightharpoonup v$ in $H$. Then by Remark \ref{remark2.1}, we have $v_n \to v$ in $L^s ( \RR^3 )$ for $2 \leq s < 6$, and $v_n \to v$ a.e. on $\RR^3 $.

We now consider the following two possible cases concerning $v $.

{\it Case 1}. If $v = 0$, then $v_n \to 0$ in $L^s ( \RR^3 )$ for $2 \leq s < 6$, and $v_n \to 0$ a.e. on $\RR^3 $. Hence it follows from \eqref{eq:2.6} and $v_n := \frac{u_n}{\|u_n\|_H^2}$ that
\begin{align}\label{eq:2.15}
\int_{\Omega_n(0,r_0)}\frac{|\tilde{F}(x,u_n)|}{\|u_n\|_H^2} dx & = \int_{\Omega_n(0,r_0)}\frac{|\tilde{F}(x,u_n)|}{|u_n|^2}|v_n|^2 dx\\\nonumber
& \leq \left(\frac{c_1}{4}r_0^2+\frac{c_2}{q}r_0^{q-2}\right)\int_{\Omega_n(0,r_0)}|v_n|^2 dx\\\nonumber
& \leq C_4 \int_{\RR^3}|v_n|^2 dx \to 0, \text{ as }n \to \infty.
\end{align}
 By(S3), we know that $\kappa > 1$. Thus, if we set $\kappa' = \kappa/(\kappa - 1 ) $, then $2 \kappa' \in ( 2 , 6 ) $. Hence it follows from (S3) and \eqref{eq:2.13} that
\begin{align}\label{eq:2.16}
\int_{\Omega_n(r_0,\infty)}\frac{|\tilde{F}(x,u_n)|}{\|u_n\|_H^2} dx & = \int_{\Omega_n(r_0,\infty)}\frac{|\tilde{F}(x,u_n)|}{|u_n|^2}|v_n|^2 dx\\\nonumber
& \leq \left[ \int_{\Omega_n(r_0,\infty)} \left( \frac{|\tilde{F}(x,u_n)|}{|u_n|^2} \right)^\kappa dx \right]^{1/\kappa} \left[ \int_{\Omega_n(r_0,\infty)} |v_n|^{2\kappa'} dx \right]^{1/\kappa'}\\\nonumber
& \leq c_4^{1/\kappa}\left[ \int_{\Omega_n(r_0,\infty)} \tilde{\mathcal{F}}(x,u_n) dx \right]^{1/\kappa}\left[ \int_{\Omega_n(r_0,\infty)} |v_n|^{2\kappa'} dx \right]^{1/\kappa'}\\\nonumber
& \leq c_4^{1/\kappa} (c+1)^{1/\kappa}\left[ \int_{\Omega_n(r_0,\infty)} |v_n|^{2\kappa'} dx \right]^{1/\kappa'}\\\nonumber
& \leq C_5 \left[ \int_{\Omega_n(r_0,\infty)} |v_n|^{2\kappa'} dx \right]^{1/\kappa'} \to 0, \text{ as }n \to \infty.
\end{align}
Combining \eqref{eq:2.15} with \eqref{eq:2.16}, we have
$$ \int_{\RR^3}\frac{|\tilde{F}(x,u_n)|}{\|u_n\|_H^2}dx = \int_{\Omega_n(0,r_0)}\frac{|\tilde{F}(x,u_n)|}{\|u_n\|_H^2} dx + \int_{\Omega_n(r_0,\infty)}\frac{|\tilde{F}(x,u_n)|}{\|u_n\|_H^2} dx \to 0 \text{ as }n \to \infty, $$
which contradicts \eqref{eq:2.14}.

{\it Case 2}. If $v\not= 0$, we set $A := \{ x \in \RR^3 : v( x )\not= 0 \} $. Then meas$ ( A ) > 0$. For a.e. $x \in A$, we have $\lim_{n\to\infty} | u_n ( x )| = \infty $. Hence $A \subset \Omega_n ( r_0 , \infty)$ for  large enough $n \in \mathbb{N}$. It follows from \eqref{eq:2.5}, \eqref{eq:2.6}, \eqref{eq:2.12} and Fatou's lemma that
\begin{align}\label{eq:2.18}
0 & = \lim_{n\to\infty}\frac{c+o(1)}{\|u_n\|_H^4}=\lim_{n\to\infty}\frac{I(u_n)}{\|u_n\|_H^4}\\\nonumber
& = \lim_{n\to\infty} \left[ \frac{1}{2\|u_n\|_H^2} + \frac{b}{4\|u_n\|_H^4} \left( \int_{\RR^3} |\nabla u|^2 dx \right)^2 - \int_{\RR^3} \frac{\tilde{F}(x,u_n)}{\|u_n\|_H^4} dx \right]\\\nonumber
& = \left[ \frac{b}{4\|u_n\|_H^4} \left( \int_{\RR^3} |\nabla u|^2 dx \right)^2 - \int_{\Omega_n(0,r_0)}\frac{|\tilde{F}(x,u_n)|}{|u_n|^4} |v_n|^4 dx - \int_{\Omega_n(r_0,\infty)}\frac{|\tilde{F}(x,u_n)|}{|u_n|^4}|v_n|^4 dx \right]\\\nonumber
& \leq \frac{b}{4} + \limsup_{n\to\infty}\int_{\Omega_n(0,r_0)}\left( \frac{c_1}{4}+\frac{c_2}{q}|u_n|^{q-4} \right)|v_n|^4dx - \liminf_{n\to\infty}\left[ \int_{\Omega_n(r_0,\infty)}\frac{|\tilde{F}(x,u_n)|}{|u_n|^4}|v_n|^4 dx \right]\\\nonumber
& \leq \frac{b}{4} + \left( \frac{c_1}{4}+\frac{c_2}{q}|r_0|^{q-4} \right) \limsup_{n\to\infty}\int_{\Omega_n(0,r_0)}|v_n|^4dx - \liminf_{n\to\infty}\left[ \int_{\Omega_n(r_0,\infty)}\frac{|\tilde{F}(x,u_n)|}{|u_n|^4}|v_n|^4 dx \right]\\\nonumber
& \leq \frac{b}{4} + C_8 - \liminf_{n\to\infty} \int_{\Omega_n(r_0,\infty)}\frac{|\tilde{F}(x,u_n)|}{|u_n|^4}|v_n|^4 dx\\\nonumber
& = \frac{b}{4} +C_8 - \liminf_{n\to\infty} \int_{\RR^3}\frac{|\tilde{F}(x,u_n)|}{|u_n|^4}[\chi_{\Omega_n(r_0,\infty)}(x)]|v_n|^4 dx\\\nonumber
& = C_9 - \int_{\RR^3}\liminf_{n\to\infty}\frac{|\tilde{F}(x,u_n)|}{|u_n|^4}[\chi_{\Omega_n(r_0,\infty)}(x)]|v_n|^4 dx \to - \infty \text{ as }n\to \infty,
\end{align}
which is a contradiction. Thus $\{ u_n \}$ is bounded in $H$. The proof is completed.
\end{proof}

\begin{lemma}\label{lemma2.6}
Assume that the conditions (V2) and (S1) hold. Then any bounded sequence $\{ u_n \}$ satisfying \eqref{eq:2.12} has a convergent subsequence in $H$.
\end{lemma}
\begin{proof}
Going if necessary to a subsequence, we may assume that $u_n \rightharpoonup u$ in $H$. Then by Remark \ref{remark2.1}, we have $v_n \to v$ in $L^s ( \RR^3 ) $, for $2 \leq s < 6$. Let us take $r \equiv 1$ in Lemma \ref{tang} and combine with $u_n \to u$ in $L^s ( \RR^3 )$ for $2 \leq s < 6$, to get
\begin{equation}\label{eq:2.21}
\lim_{n\to\infty}|\tilde{f}(x,u_n)-\tilde{f}(x,u)||u_n-u|dx \to 0, \text{ as }n\to\infty.
\end{equation}
We observe that
 \begin{equation}\label{eq1}
    \langle I' (u_n) - I'(u), u_n - u \rangle \to 0, \text{ as }n \to \infty,
 \end{equation}
and we have
\begin{align}\label{eq2.7}
& \langle I'(u_n) - I'(u), u_n - u \rangle \\\nonumber
= & \int_{\RR^3} \tilde{V}(x) |u_n - u|^2 dx + \left( 1 + b\int_{\RR^3} |\nabla u_n|^2 dx \right) \int_{\RR^3} \nabla u_n \cdot \nabla (u_n -u) dx \\\nonumber
& - \left( 1 + b\int_{\RR^3} |\nabla u|^2 dx \right) \int_{\RR^3} \nabla u \cdot \nabla (u_n -u) dx - \int_{\RR^3} [f(x,u_n) - f(x,u)] (u_n - u)dx\\\nonumber
= & \|u_n - u\|_H^2 + \left( 1 + b\int_{\RR^3} |\nabla u_n|^2 dx \right)  \int_{\RR^3} |\nabla (u_n -u)|^2 dx\\\nonumber
& - \left( \int_{\RR^3} |\nabla u|^2 dx - \int_{\RR^3} |\nabla u_n|^2 dx \right) \int_{\RR^3 } \nabla u \cdot \nabla (u_n -u) dx\\\nonumber
& - \int_{\RR^3} [f(x,u_n) - f(x,u)] (u_n - u)dx\\\nonumber
& \geq \|u_n - u\|_H^2 - b\left( \int_{\RR^3} |\nabla u|^2 dx - \int_{\RR^3} |\nabla u_n|^2 dx \right)\\\nonumber
& - \int_{\RR^3 } \nabla u \cdot \nabla (u_n -u) dx - \int_{\RR^3} [f(x,u_n) - f(x,u)] (u_n - u)dx.
\end{align}
Then \eqref{eq2.7} implies that
\begin{align}\label{eq2.8}
\|u_n - u\|_H^2 &  \leq  \langle I'(u_n) - I'(u), u_n - u \rangle\\\nonumber
& + b\left( \int_{\RR^3} |\nabla u|^2 dx - \int_{\RR^3} |\nabla u_n|^2 dx \right) \int_{\RR^3 } \nabla u \cdot \nabla (u_n -u) dx\\\nonumber
& + \int_{\RR^3} [f(x,u_n) - f(x,u)] (u_n - u)dx.
\end{align}
Define the functional $h_u : H \to \RR$ by
$$ h_u(v) = \int_{\RR^3} \nabla u \cdot \nabla v dx, \forall v \in H. $$
Obviously, $h_u$ is a linear functional on $H$. Furthermore,
$$ |h_u(v)| \leq \int_{\RR^3} |\nabla u \cdot \nabla v| dx \leq \|u\|_H \|v\|_H, $$
which implies that $h_u$ is bounded on $H$. Hence $h_u \in H^*$. Since $u_n \rightharpoonup u$ in $H$, we have $\lim_{n \to \infty}  h_u (u_n ) = h_u (u)$, that is,
$$ \int_{\RR^3} \nabla u \cdot \nabla (u_n - u)dx \to 0\quad \mbox{as $n \to \infty$}.$$
Consequently, by $v_n \to v$ in $L^s ( \RR^3 ) $, for $2 \leq s < 6$ and the boundedness of $\{u_n \}$, we obtain
\begin{equation}\label{eq2.9}
b\left( \int_{\RR^3} |\nabla u|^2 dx - \int_{\RR^3} |\nabla u_n|^2 dx \right) \int_{\RR^3 } \nabla u \cdot \nabla (u_n -u) dx \to 0, \quad n \to + \infty.
\end{equation}
Consequently, \eqref{eq:2.21}, \eqref{eq1}, \eqref{eq2.8}, \eqref{eq2.9} imply that
$$ u_n \to u \text{ in }H \text{ as }n \to \infty. $$
This completes the proof.
\end{proof}

\begin{proof}[Proof of Theorem \ref{thm1.1}]
To complete the proof of the main result, we need to consider the following two steps.

{\it Step 1.} We first show that there exists a function $u_0 \in H$ such that $I' ( u_0 ) = 0$ and $I ( u_0 ) < 0$. Let $r_0 = 1$. For any $| u | \geq 1$, from (S2), we have
\begin{equation}\label{eq:3.1}
\tilde{F}(x,u_n)\geq c_3|u|^\sigma > 0.
\end{equation}
By (S1), for a.e. $x \in \RR^3$ and $0 \leq | u | \leq 1$, there exists $M > 0$ such that
$$ \left| \frac{\tilde{f}(x,u)u}{u^2} \right| \leq \left| \frac{(c_1|u|^3+c_2|u|^{q-1})|u|}{|u|^2} \right| \leq M, $$
which implies that
$$ \tilde{f}(x,u)u \geq - M|u|^2. $$
Using the equality $\tilde{F}(x,u)=\int_0^1\tilde{f}(x,tu)dt$, for a.e. $x\in\RR^3$ and $0\leq|u|\leq1$, we obtain
\begin{equation}\label{eq:3.4}
\tilde{F}(x,u)> -\frac{1}{2}M|u|^2.
\end{equation}
In view of \eqref{eq:3.1} and \eqref{eq:3.4}, we have for a.e. $x \in \RR^3$ and all $u \in \RR$ that
$$ \tilde{F}(x,u) \geq -\frac{1}{2}M|u|^2+c_3|u|^\sigma. $$
Therefore we have
\begin{equation}\label{eq:3.6}
\tilde{F}(x,t\psi)\geq - \frac{1}{2}Mt^2|\psi|^2+t^\sigma c_3|\psi|^\sigma.
\end{equation}
Combing \eqref{eq:2.5} with \eqref{eq:3.6}, we get
\begin{align*}
I(tu)& = \frac{t^2}{2}\|u\|_H^2 + \frac{bt^4}{4}\left( \int_{\RR^3} |\nabla u|^2 dx \right)^2 - \int_{\RR^3}\tilde{F}(x,tu)dx\\
& \leq \frac{t^2}{2}\|u\|_H^2 + \frac{bt^4}{4}\left( \int_{\RR^3} |\nabla u|^2 dx \right)^2 + \frac{t^2M}{2}\int_{\RR^3}|u|^2 dx - t^\sigma c_3 \int_{\RR^3}|u|^\sigma dx.
\end{align*}
Since $\sigma \in ( 0 , 2 ) $, for  small enough $t$ we infer that $I ( tu ) < 0$. Thus we obtain
$$ c_0 = \inf\{I(u):u\in\bar{B}_\rho\}<0, $$
where $\rho > 0$ is given by Lemma \ref{lemma2.2} and $B_\rho = \{ u \in H : \| u \|_H < \rho\} $. By the Ekeland variational principle, there exists a sequence $\{ u_n \} \subset B_\rho$ such that
$$ c_0 \leq I(u_n) \leq c_0 + \frac{1}{n}, $$
and
$$ I(w)\geq I(u_n)-\frac{1}{n}\|w-u_n\|_H, $$
for all $w \in B_\rho $. Then, following the idea of \cite{MR1400007}, we can show that $\{u_n\}$ is a bounded Cerami sequence of $I$. Therefore, Lemma \ref{lemma2.6} implies that there exists a function $u_0 \in H$ such that $I' ( u_0 ) = 0$ and $I ( u_0 ) = c_0 < 0$.

{\it Step 2.} We now show that there exists a function $\tilde{u}_0 \in H$ such that $I'(\tilde{u}_0 ) = 0$ and $I (\tilde{u}_0 ) = \tilde{c}_0 > 0$. By Lemmata \ref{lemma2.2}, \ref{lemma2.3} and \ref{lemma2.1}, there is a sequence $\{ u_n \} \in H$ satisfying \eqref{eq:2.12}. Moreover, Lemma \ref{lemma2.4} and \ref{lemma2.6} shows that this sequence has a convergent subsequence and is bounded in $H$. So, we complete the Step 2.

Therefore, combining the above two steps and Lemma \ref{lemma1.1}, we complete the proof of Theorem \ref{thm1.1}.
\end{proof}

\section*{Acknowledgments}
 This research was supported by the Slovenian Research Agency grants P1-0292, J1-5435, and J1-6721, the Romanian National Authority for Scientific Research CNCS-UEFISCDI grant PN-II-ID-PCE-2011-3-0195, and the Research Fund of Chongqing Technology and Business University grant 2015-56-09-1552007. 

\footnotesize

\bibliographystyle{abbrv} 
\bibliography{kirchhoff}

\end{document}